\documentclass[12pt]{amsart}
\usepackage{amscd,amsmath,amsthm,amssymb}
\usepackage[left]{lineno}
\usepackage{pstcol,pst-plot,pst-3d}
\usepackage{color}
\usepackage{pstricks}
\usepackage{stmaryrd}
\usepackage[utf8]{inputenc}
\usepackage{pstricks-add}
\usepackage[dvips]{graphicx}
\usepackage{epstopdf}
\usepackage{graphicx}
\newpsstyle{fatline}{linewidth=1.5pt}
\newpsstyle{fyp}{fillstyle=solid,fillcolor=verylight}
\definecolor{verylight}{gray}{0.97}
\definecolor{light}{gray}{0.9}
\definecolor{medium}{gray}{0.85}
\definecolor{dark}{gray}{0.6}

 \def\G{{\mathcal G}}

 \def\opn#1#2{\def#1{\operatorname{#2}}} 
 \opn\chara{char} \opn\length{\ell} \opn\pd{pd} \opn\rk{rk}
 \opn\projdim{proj\,dim} \opn\injdim{inj\,dim} \opn\rank{rank}
 \opn\depth{depth} \opn\grade{grade} \opn\height{height}
 \opn\embdim{emb\,dim} \opn\codim{codim}
 
 \opn\Tr{Tr} \opn\bigrank{big\,rank}
 \opn\superheight{superheight}\opn\lcm{lcm}
 \opn\trdeg{tr\,deg}
 \opn\reg{reg} \opn\lreg{lreg} \opn\ini{in} \opn\lpd{lpd}
 \opn\size{size} \opn\sdepth{sdepth}
 \opn\link{link}\opn\fdepth{fdepth}\opn\lex{lex}
 \opn\tr{tr}
 \opn\type{type}
 \opn\gap{gap}
 \opn\arithdeg{arith-deg}
 \opn\div{div} \opn\Div{Div} \opn\cl{cl} \opn\Cl{Cl}
 \opn\Spec{Spec} \opn\Supp{Supp} \opn\supp{supp} \opn\Sing{Sing}
 \opn\Ass{Ass} \opn\Min{Min}\opn\Mon{Mon}
 \opn\Ann{Ann} \opn\Rad{Rad} \opn\Soc{Soc}
 \opn\Im{Im} \opn\Ker{Ker} \opn\Coker{Coker} \opn\Am{Am}
 \opn\Hom{Hom} \opn\Tor{Tor} \opn\Ext{Ext} \opn\End{End}
 \opn\Aut{Aut} \opn\id{id}
 
 \opn\nat{nat}
 \opn\pff{pf}
 \opn\Pf{Pf} \opn\GL{GL} \opn\SL{SL} \opn\mod{mod} \opn\ord{ord}
 \opn\Gin{Gin} \opn\Hilb{Hilb}\opn\sort{sort}
 \opn\PF{PF}\opn\Ap{Ap}
 \opn\mult{mult}
 \opn\bight{bight}
 \opn\aff{aff}
 \opn\relint{relint} \opn\st{st}
 \opn\lk{lk} \opn\cn{cn} \opn\core{core} \opn\vol{vol}  \opn\inp{inp} \opn\nilpot{nilpot}
 \opn\link{link} \opn\star{star}\opn\lex{lex}\opn\set{set}
 \opn\width{wd}
 \opn\Fr{F}
 \opn\QF{QF}
 \opn\G{G}
 \opn\type{type}\opn\res{res}
 \opn\conv{conv}
 \opn\Ind{Ind}
 \opn\gr{gr}
 
 \def\pot#1#2{#1[\kern-0.28ex[#2]\kern-0.28ex]}

 \opn\dirlim{\underrightarrow{\lim}}
 \opn\inivlim{\underleftarrow{\lim}}
 \def\Implies{\ifmmode\Longrightarrow \else
         \unskip${}\Longrightarrow{}$\ignorespaces\fi}
 \def\implies{\ifmmode\Rightarrow \else
         \unskip${}\Rightarrow{}$\ignorespaces\fi}
 \def\iff{\ifmmode\Longleftrightarrow \else
         \unskip${}\Longleftrightarrow{}$\ignorespaces\fi}

 \let\:=\colon
\newtheorem{theorem}{Theorem}[section]

\newtheorem{lemma}[theorem]{Lemma}
\newtheorem{proposition}[theorem]{Proposition}
\newtheorem{definition}[theorem]{Definition}

\newtheorem{remark}[theorem]{Remark}
\newtheorem{remarks}[theorem]{Remarks}

 \let\epsilon\varepsilon
 \let\kappa=\varkappa
 \def\qed{\ifhmode\textqed\fi
       \ifmmode\ifinner\quad\qedsymbol\else\dispqed\fi\fi}
 \def\textqed{\unskip\nobreak\penalty50
        \hskip2em\hbox{}\nobreak\hfil\qedsymbol
        \parfillskip=0pt \finalhyphendemerits=0}
 \def\dispqed{\rlap{\qquad\qedsymbol}}
 
 \opn\dis{dis}
 \def\pnt{{\raise0.5mm\hbox{\large\bf.}}}
 
 \opn\Lex{Lex}

\begin{document}

\title {On a generalization of Roman domination with more legions}

\author {Fahimeh Khosh-Ahang}

\address{Fahimeh Khosh-Ahang Ghasr, Department of Mathematics, School of Science, Ilam University,
P.O.Box 69315-516, Ilam, Iran}
\email{fahime$_{-}$khosh@yahoo.com and f.khoshahang@ilam.ac.ir}

\dedicatory{ }

\begin{abstract}
In this note, we generalize the concepts of (perfect) Roman and Italian dominations to (perfect) strong Roman and Roman $k$-domination for arbitrary positive integer $k$. These generalizations cover some of previous ones. After some comparison of their domination numbers, as a first study of these concepts, we study the (perfect, strong, perfect strong) Roman $k$-domination numbers of complete bipartite graphs.
\end{abstract}

\thanks{}




\maketitle

\setcounter{tocdepth}{1}

\section{Introduction}
Roman domination is initiated after the incipient paper of Stewart (\cite{Stewart}) in 1999, even though it has some historical roots in some older papers too. In fact it is a strategy of defending the Roman Empire. It can be mentioned that \cite{Cokayane} is the first paper in which this defense strategy is studied in graph theory profoundly. After that a various types of generalizations are introduced and studied spreadly such as weak Roman domination \cite{Henning+hedetniemi}, perfect Roman domination \cite{Klostermeyer2}, Italian Domiantion \cite{Chellali}, perfect Italian domination \cite{perfect}, double Roman domination \cite{double} and etc. Also a couple of survays are written in this context (cf. \cite{Klostermeyer1}, \cite{Klostermeyer2}, \cite{Chellali2}, \cite{Chellali3} and \cite{Chellali1}).

In Roman domination problem, two legions are needed to defend any vertex, and the problem is supporting the vertices with no legions by their neigbors with two legions. This domination is called perfect if each vertex with no legions has exactly one neigbor with two legions. In double Roman domination problem, which is a kind of its generalization, at least three legions are needed to defend any vertex, and the problem is supporting the vertices with zero or one legions by its neigbors with at least two legions. So, it can be a natural generalization to ask about Roman $k$-domination problem. That is $k$ legions are needed to defend any vertex, and so considering the best ways to support the vertices with the number of legions less than $k/2$ with their neigbors is the main problem of this kind of domination. We can also bring a similar argument for Italian domination (Roman $\{2\}$-domination). In Italian domination although, as in Roman domination, two legions are needed to defend any vertex, but the problem is supporting the vertices with no legions by all of their neigbors. In this regard, generalizing these kinds of dominations for arbitrary positive integers $k$ makes it more useful. Because in different situations it may need more than $2$ or $3$ of something, say legions. For instance, in covid-19 pandemic, the shortages of hospital beds, medical oxygen, nurses and etc. in many countries are notable. So for example when we know that the safe amount of something in a city is $k$, arranging the available facilities in the cities by holding more facilities in major points and considering their transfer between neigbor cities, can help to pass from these shortages with less expenses.  Of course it is not that simple and it depends to other factors too.

 In this note, firstly in Section $2$, we generalize (perfect) Roman and Italian domination to (perfect)  strong Roman and Roman $k$-domination and bring some elementary results about their domination numbers and their comparision. Then in the next section, we gain these $k$-domination numbers for complete bipartite graph $K_{m,n}$ by discussion on the vlues of $m, n$ and $k$ explicitly, in Theorems \ref{thm1}, \ref{thm2}, \ref{thm3} and \ref{thm4}.

Throughout $k$ is a positive integer and $G=(V,E)$ is a finite connected simple graph unless otherwise is stated. For each vertex $u\in V$ we use the notion of $N_G(u)$ for the set of all neigbors of $u$ in $G$. Also $N_G[u]=\{u\}\cup N_G(u)$. We also use the notion $K_{m,n}$ for the complete bipartite graph in which the size of its partite sets is $m$ and $n$.

\section{Some results on Roman $k$-domination}
We begin our section by introducing the generalization of Roman domination for more number of legions as follows.
\begin{definition}\label{2.1}
Let $G=(V,E)$ be a finite simple graph, $k$ be a positive integer and $f: V\rightarrow \{0, 1, \dots , k\}$ be a map. For each vertex $u$ of $G$, $f(u)$ is called the weight of $u$ by $f$ or briefly weight of $u$ if there is no ambiguity about $f$. If we denote the set of vertices $u$ of $G$ with $f(u)=i$ by $V_i$, then we sometimes write $f=(V_0, V_1, \dots, V_k)$. The weight of $f$ is $w(f)=\sum_{u\in V}f(u)$.

 $f$ is called a (perfect) Roman $k$-dominating function if $f(u)=k$ for all isolated vertices $u$ of $G$, and also for all $u\in V$ with $f(u)<k/2$ we have $\sum_{v\in N_G[u]}f(v)\geq k$ ($\sum_{v\in N_G[u]}f(v)= k$).
$f$ is called a (perfect) strong Roman $k$-dominating function if  $f(u)=k$ for all isolated vertices $u$ of $G$, and also for all $u\in V$ with $f(u)<k/2$ we have $f(u)+\sum_{v\in N_G(u), f(v)>k/2}f(v)\geq k$ ($f(u)+\sum_{v\in N_G(u), f(v)>k/2}f(v)= k$). Also we define Roman $ k$-domination number of $G$ ($\gamma_k(G)$) respectively strong Roman $k$-domination number of $G$ ($\gamma_k^s(G)$), perfect Roman $k$-domination number of $G$ ($\gamma_k^p(G)$) and perfect strong Roman $k$-domination number of $G$ ($\gamma_k^{ps}(G)$) as follows: 
$$\gamma_k(G)=\min \{w(f) \ | \ f \text{ is a Roman } k \text{-dominating function of } G\},$$
$$\gamma^s_k(G)=\min \{w(f) \ | \ f \text{ is a strong Roman } k \text{-dominating function of } G\},$$
$$\gamma^p _k(G)=\min \{w(f) \ | \ f \text{ is a perfect Roman } k \text{-dominating function of } G\},$$
$$\gamma^{ps}_k(G)=\min \{w(f) \ | \ f \text{ is a perfect strong Roman } k \text{-dominating function of } G\}.$$
 Every   Roman $k$-dominating function (resp. perfect, strong and perfect strong Roman $k$-dominating function) $f$ with $w(f)=\gamma_k(G)$ (resp. $w(f)=\gamma_k^p(G)$, $w(f)=\gamma_k^s(G)$ and $w(f)=\gamma_k^{ps}(G)$) is called a $\gamma_k$-function (resp. $\gamma_k^p(G)$-function, $\gamma_k^s(G)$-function, $\gamma_k^{ps}(G)$-function) of $G$.
\end{definition} 
The following remarks immediately follows from Definition \ref{2.1}.
\begin{remarks}\label{2.2}
\begin{itemize}
\item[1)] It can be easily seen that strong Roman $k$-domination for $k=2$ and $k=3$ are respectively Roman domination and double Roman domination. Also, Roman $2$-domination is precisely Italian domination. 

\item[2)] If $f=(V_0, \dots, V_k)$, then $V=\bigcup_{i=0}^kV_i$ and so $w(f)=\sum_{i=1}^k i|V_i|$.

\item[3)] If $G=K_n$ is the complete graph with $n$ vertices, then by assigning $k$ to one vertex and zero to other vertices, we get to a (perfect, strong, perfect strong) Roman $k$-dominating function. So $k$ is the minimum value of $\gamma_k(G)$, $\gamma_k^p(G)$, $\gamma^s_k(G)$ and $\gamma_k^{ps}(G)$.

\item[4)] It is easily seen that every strong Roman $k$-dominating function is a Roman $k$-dominating function, every perfect Roman $k$-dominating function is a Roman $k$-dominating function and every perfect strong Roman $k$-dominating function is a strong Roman $k$-dominating function. 

\item[5)] If $G$ has $t$ connected components, then clearly, $tk$ is the minimum value of $\gamma_k(G), \gamma_k^p(G), \gamma_k^s(G)$ and $\gamma_k^{ps}(G)$

\item[6)]  It is clear that if $k$ is even, then by assigning the weight $k/2$ to all vertices and if $k$ is odd, then by assigning the weight $(k+1)/2$ to all vertices, one can get to a Roman $k$-dominating function which is also strong, perfect and perfect srtong. So
$$k\leq \gamma_k(G)\leq \gamma_k^p(G)\leq  \left\lbrace\begin{array}{l}
\frac{|V|k}{2}; \ \ \ k \text{ is even}, \\ 
\frac{|V|(k+1)}{2}; \ \ k \text{ is odd}.
\end{array}\right. $$

$$k\leq \gamma_k(G)\leq  \gamma_k^s(G)\leq \gamma_k^{ps}(G) \leq  \left\lbrace\begin{array}{l}
\frac{|V|k}{2}; \ \ \ k \text{ is even}, \\ 
\frac{|V|(k+1)}{2}; \ \ k \text{ is odd}.
\end{array}\right. $$

\item[7)] If $H$ is a spanning subgraph of $G$ (that is $G$ can be gained by adding edges to $H$), then certainly every (strong) Roman $k$-dominating function of $H$ is also a (strong) Roman $k$-dominating function of $G$. So, $\gamma_k(G) \leq \gamma_k(H)$ ($\gamma^s_k(G) \leq \gamma^s_k(H)$). In particular, if $G$ has a spanning subgraph $H$ such that $\gamma_k(H)=k$ ($\gamma^s_k(H)=k$), then we also have $\gamma_k(G)=k$ ($\gamma^s_k(G)=k$).

\item[8)] If $f$ is a strong Roman $k$-dominating function and $u$ is a vertex with $f(u)<k/2$, then there should exist at least one vertex $v\in N_G(u)$ with $f(v)>k/2$. If $f$ is also perfect, then there should exist exactly one vertex $v\in N_G(u)$ with $f(v)>k/2$.
\end{itemize}
\end{remarks}

In the next result, we use the technique of the proof of Lemma 1 in \cite{perfect} to show that the last equality in Remarks \ref{2.2}(6) for $\gamma_k^p(G)$ may occur for some bipartite graph.
\begin{proposition}
$$\sup \{\gamma_k^p(G) \ | \ G \text{ is a bipartite graph} \} =
\left\lbrace\begin{array}{l}
\frac{|V|k}{2}; \ \ \ k \text{ is even}, \\ 
\frac{|V|(k+1)}{2}; \ \ k \text{ is odd}.
\end{array}\right. $$
\end{proposition}
\begin{proof}
Suppose that $H=K_{m,m^\ell}$ is a complete bipartite graph with $A, B$ as its partite sets such that $|A|=m\geq 3$ and $|B|=m^\ell$ for some $\ell\geq 2$. Also assume that $G=(V,E)$ is a graph which is obtained by attaching three pendant vertices $x_{1,v}, x_{2,v}, x_{3,v}$ to each vertex $v$ of $A$. So $G$ is a bipartite graph with $|V|=m^\ell+4m$. Assume that $f$ is an arbitrary perfect Roman $k$-dominating function of $G$.

 Let $k$ be even. If $f(v)<k/2$ for some $v\in A$, then for $i=1, 2, 3$,  $f(x_{i,v})\geq k/2$, because else we should have $k=f(x_{i,v})+f(v)<k/2+k/2=k$ which is a contradiction. Hence 
$$k=f(v)+\sum_{u\in N_G(v)}f(u)\geq f(v)+f(x_{1,v})+f(x_{2,v})+f(x_{3,v})\geq 3k/2,$$
which is a contradiction. So we should have $f(v)\geq k/2$ for all $v\in A$. Now if $f(x_{i,v})<k/2$ for some $i=1, 2, 3$, then $f(x_{i,v})+f(v)=k$ and else $$f(v)+\sum_{i=1}^3f(x_{i,v})\geq 4k/2>k.$$
Hence for each $v\in A$, we have $f(v)\geq k/2$ and  $f(v)+\sum_{i=1}^3f(x_{i,v})\geq k$. So since 
$|A|=m\geq 3$, for each vertex $v\in B$, we have $f(v)+\sum_{u\in N_G(v)}f(u)\geq \sum_{u\in A}f(u)\geq 3k/2>k$. Thus $f(v)\geq k/2$ for all vertices $v\in B$. Therefore
\begin{align*}
w(f)&=\sum_{v\in A}(f(v)+\sum_{i=1}^3f(x_{i,v}))+\sum_{v\in B}f(v)\\
&\geq |A|k+|B|k/2\\
&= mk+m^\ell k/2.
\end{align*}
On the other hand, $|V|=m^\ell +4m=m^\ell(1+4/m^{\ell -1})$. So, 
$$\frac{w(f)}{|V|}\geq \frac{m^\ell k(1/m^{\ell -1}+1/2)}{m^\ell(1+4/m^{\ell -1})}.$$
Hence for enough large $m$, we have $w(f)/|V|\geq k/2$ and so $w(f)\geq \frac{|V|k}{2}$. Therefore in view of Remarks \ref{2.2}(6), $\gamma_k^p(G)=\frac{|V|k}{2}$, for enough large $m$.

 Let $k$ be odd. By similar argument to above, we have
 \begin{align*}
 &f(v)+\sum_{i=1}^3f(x_{i,v})\geq k; \ \forall v\in A,\\
 &f(v)\geq (k+1)/2; \ \forall v\in B.
 \end{align*}
 
 Therefore
\begin{align*}
w(f)&=\sum_{v\in A}(f(v)+\sum_{i=1}^3f(x_{i,v}))+\sum_{v\in B}f(v)\\
&\geq |A|k+|B|(k+1)/2\\
&= mk+m^\ell (k+1)/2.
\end{align*}
On the other hand, $|V|=m^\ell+4m=m^\ell (1+4/m^{\ell -1})$. So, 
$$\frac{w(f)}{|V|}\geq \frac{m^\ell (k/m^{\ell -1}+(k+1)/2)}{m^\ell (1+4/m^{\ell -1})}.$$
Hence for enough large $m$, we have $w(f)/|V|\geq (k+1)/2$ and so $w(f)\geq \frac{|V|(k+1)}{2}$. Therefore in view of Remarks \ref{2.2}(6),
$\gamma_k^p(G)=\frac{|V|(k+1)}{2}$ for enough large $m$.
\end{proof}

The following lemma is needed for our next theorem.
\begin{lemma}(Compare \cite[Proposition 1]{double}.) \label{2.3}
Suppose that $f=(V_0, V_1, \dots , V_k)$ is a $\gamma^s_k$-function of $G$. Then we may assume that for each $0<i<k/2$, we have $|V_i|=0$. 
\end{lemma}
\begin{proof}
Suppose that $f=(V_0, V_1, \dots , V_k)$ is a $\gamma^s_k$-function of $G$ and $u\in V_i$ for some $0<i <k/2$. Then
 $f(u)+\sum_{v\in N_G(u), f(v)>k/2}f(v)\geq k$. Now if $\sum_{v\in N_G(u), f(v)>k/2}f(v)\geq k$, then we define $g(u)=0$ and $g(v)=f(v)$ for all $v\in V\setminus \{u\}$. It is obvious that $g$ is a strong Roman $k$-dominating function with $w(g)<w(f)$, which is a contradiction, since $f$ is a $\gamma^s_k$-function. Else, $\sum_{v\in N_G(u), f(v)>k/2}f(v)=j<k$. Then for all $v\in N(u)$ with $f(v)>k/2$, we have $f(v)\leq j <k$ and so $f(v)+k-j\leq k$. In this case suppose that $v_0\in N(u)$ with $f(v_0)>k/2$ (note that by Remarks \ref{2.2}(8) such $v_0$ exists). Then set $g(u)=0$, $g(v_0)=f(v_0)+k-j$ and $g(v)=f(v)$ for all $v\in V\setminus \{u, v_0\}$. Since $i+j\geq k$, we have $w(g)=w(f)-i+k-j\leq w(f)$. So since $f$ is a $\gamma^s_k$-function and $g$ is clearly a strong Roman $k$-dominating function, we should have $w(g)=w(f)$. By repeating this process for all $u$ with $0<f(u)<k/2$, we get to a $\gamma_k^s$-function of $G$ with the desired property.
\end{proof}

The next theorem gives some relations between $\gamma_k^s(G)$ and $\gamma_{k+1}^s(G)$.
\begin{theorem}\label{2.4}
For each $\gamma_k^s$-function $f=(V_0, \dots , V_k)$ of $G$ we have 
$$\gamma_{k+1}^s(G)\leq \gamma_k^s(G)+|V|-|V_0|\leq 2\gamma_k^s(G).$$
\end{theorem}
\begin{proof}
We consider the case that $k$ is even. When $k$ is odd, we can proceed in a similar manner by replacing $(k+1)/2$ in stead of $k/2$.

Suppose that $f=(V_0, \dots , V_k)$ is a $\gamma _k^s$-function of $G$. In view of Lemma \ref{2.3}, we may assume that $V_i=\emptyset$ for all $0< i <k/2$. So
 $$|V|-|V_0|=\sum_{i=k/2}^k |V_i|, w(f)=\sum_{i=k/2}^ki|V_i|.$$
  Now set $V'_0=V_0$, $V'_i=\emptyset$ for all $1\leq i \leq k/2$, and $V'_i=V_{i-1}$ for all $k/2<i\leq k+1$. Then $f'=(V'_0, \dots , V'_{k+1})$ is a strong Roman $(k+1)$-dominating function of $G$. Because if $u$ is a vertex with $f'(u)<(k+1)/2$, then $u\in V'_i$ for some $0\leq i \leq k/2$. Since $V'_i=\emptyset$ for $1\leq i\leq k/2$, we should have $f'(u)=f(u)=0$. Now since $f$ is a strong Roman $k$-dominating function, we have by Remarks \ref{2.2}(8),
 \begin{align*}
 f'(u)+\sum_{v\in N_G(u), f'(v)>(k+1)/2}f'(v)&=f(u)+\sum_{v\in N_G(u), f'(v)>k/2}f'(v)\\
 &=f(u)+\sum_{v\in N_G(u), f(v)>k/2-1}(f(v)+1)\\
 &\geq f(u)+\sum_{v\in N_G(u), f(v)>k/2}(f(v)+1)\\
 &\geq f(u)+\sum_{v\in N_G(u), f(v)>k/2}f(v)+1\\
 &\geq k+1.
 \end{align*}
 
  Therefore
\begin{align*}
\gamma_{k+1}^s(G)&\leq w(f')\\
&=\sum_{i=1}^{k+1}i|V'_i|\\
&=\sum_{i=k/2+1}^{k+1}i|V_{i-1}|\\
&=\sum_{i=k/2+1}^{k+1}(i-1)|V_{i-1}|+\sum_{i=k/2+1}^{k+1}|V_{i-1}|\\
&=\sum_{i=k/2}^{k}i|V_i|+\sum_{i=k/2}^{k}|V_i|\\
&=w(f)+|V|-|V_0|\\
&=\gamma _k^s(G)+|V|-|V_0|.
\end{align*} 
Also since 
$$|V|-|V_0|=\sum_{i=1}^k|V_i| \leq \sum_{i=1}^ki|V_i|=\gamma_k^s(G),$$
the second inequality also holds.
\end{proof}

Let $f=(V_0,V_1, V_2)$ be a $\gamma_2^s$-function of $G$. Then Theorem \ref{2.4} ensures that
\begin{align*}
\gamma _3^s(G)&\leq \gamma_2^s(G)+|V|-|V_0|\\
&=w(f)+|V|-|V_0|\\
&=(|V_1|+2|V_2|)+(|V_0|+|V_1|+|V_2|)-|V_0|\\
&=2|V_1|+3|V_2|\leq 2\gamma _2^s(G),
\end{align*}
which has been gained in \cite[Proposition 2 and Corollary 3]{double}.

Furthermore, note that $\gamma_{k+1}^s(G)<2\gamma_k^s(G)$ unless $k=2$, which is explained in \cite[Corollary 3]{double}. Because if $\gamma_{k+1}^s(G)=2\gamma_k^s(G)$, then
$$\gamma_{k+1}^s(G)=2\gamma_k^s(G)=\gamma_k^s(G)+\sum_{i=1}^ki|V_i|.$$
Also by Theorem \ref{2.4}, we have 
$$\gamma_{k+1}^s(G)\leq \gamma_k^s(G)+|V|-|V_0|=\gamma_k^s(G)+\sum_{i=1}^k|V_i|,$$
and hence
$$\gamma_k^s(G)+\sum_{i=1}^ki|V_i|\leq \gamma_k^s(G)+\sum_{i=1}^k|V_i|,$$
which implies $\sum_{i=1}^k (i-1)|V_i|\leq 0$. So $|V_i|=0$ for all $i=2, \dots , k$. Also by Remarks \ref{2.2}(8), $|V_i|=0$ for all $i<k/2$. Hence the only case is $k=2$ which implies that $|V|=|V_1|$ and else $|V|=0$.

We end this section by the following result which compares $\gamma_k(G)$ and $\gamma_{k+1}(G)$ together.

\begin{proposition}
For each $k\in \mathbb{N}$ and each $\gamma_k$-function $f=(V_0, \dots , V_k)$ of $G$ we have 
$$\gamma_{k+1}(G)\leq \gamma_k(G)+\sum_{i=0}^r|V_i|\leq 2\gamma_k(G)+|V_0|,$$
where $r=k/2$ if $k$ is even and $r=(k-1)/2$ if $k$ is odd.
\end{proposition}
\begin{proof}
Suppose that  $f=(V_0, \dots , V_k)$ is a $\gamma_k$-function. 

Let $k$ be even. If $f(u)=i$ for some $0\leq i \leq k/2$, then we set $f'(u)=i+1$, else we set $f'(u)=f(u)$. Note that $f'=(\emptyset, V_0, V_1, \dots, V_{k/2-1}, V_{k/2}\cup V_{k/2+1}, V_{k/2+2}, \dots, V_k, \emptyset )$. Hence if $f'(u)<(k+1)/2$, then $f(u)<k/2$. Now since $f$ is a $\gamma_k$-function, 
$$f'(u)-1+\sum_{v\in N_G(u)}f'(v)=f(u)+\sum_{v\in N_G(u)}f'(v)\geq f(u)+\sum_{v\in N_G(u)}f(v)\geq k.$$
So, $f'(u)+\sum_{v\in N_G(u)}f'(v)\geq k+1$. This shows that $f'$ is a Roman $(k+1)$-dominating function. So
\begin{align*}
\gamma_{k+1}(G)&\leq w(f')\\
&=\sum_{i=0}^{k/2}(i+1)|V_i|+\sum_{i=k/2+1}^ki|V_i|\\
&=\sum_{i=0}^k i|V_i|+\sum_{i=0}^{k/2}|V_i|\\
&=\gamma_k(G)+\sum_{i=0}^{k/2}|V_i|\\
&\leq \gamma_k(G)+|V_0|+\sum_{i=1}^ki|V_i|\\
&=2\gamma_k(G)+|V_0|.
\end{align*}
Let $k$ be odd. If $f(u)=i$ for some $0\leq i \leq (k-1)/2$, then we set $f'(u)=i+1$, else we set $f'(u)=f(u)$. Note that $f'=(\emptyset, V_0, V_1, \dots, V_{(k-3)/2}, V_{(k-1)/2}\cup V_{(k+1)/2}, V_{(k+3)/2}, \dots, V_k, \emptyset )$. By similar argument to the case ``$k$ is even" the proof will be completed.
\end{proof}

\section{Roman $k$-domination of complete bipartite graphs}
In this section, we consider the complete bipartite graph $K_{m,n}$ with vertex set $V$ such that $A=\{u_1, \dots , u_m\}$ and $B=\{v_1, \dots , v_n\}$ are two partite sets. It is clear that by assigning $k$ to one vertex in each partite set and zero to other vertices, one can get to a (perfect, strong, perfect strong) Roman $k$-dominating function. So, in view of Remarks \ref{2.2}(6), we have
$$k\leq \gamma_k(K_{m,n})\leq \gamma_k^p(K_{m,n})\leq 2k,$$
and
$$k\leq \gamma_k(K_{m,n})\leq  \gamma_k^s(K_{m,n})\leq \gamma_k^{ps}(K_{m,n}) \leq 2k.$$
In the first main result of this section we characterize the perfect strong Roman $k$-domination number of complete bipartite graph in all different circumstances as follows.
\begin{theorem}\label{thm1}
$$\gamma_k^{ps}(K_{m,n})=\left\lbrace \begin{array}{ll}
k; & m=1, n\geq 1 \\ 
3k/2; & k \textrm{ is even}, m=2, n\geq 2 \\ 
2k; & \textrm{otherwise} 
\end{array} \right.$$
\end{theorem}
\begin{proof}
If $m=1$ and $n\geq 1$, then by setting $f(u_1)=k$ and $f(v_i)=0$ for all $1\leq i \leq n$, $f$ will be a perfect strong Roman $k$-dominating function. So $\gamma_k^{ps}(K_{1,n})=k$. 

 Now suppose that $m,n\geq 2$ and $f$ is a perfect strong Roman $k$-dominating function. The following cases may occur:

Case I. $f(u)\geq k/2$ for all $u\in V$. Then since $|V|\geq 4$, 
$$w(f)\geq 2k.$$

Case II. $f(u)\geq k/2$ for all $u\in A$ and $f(v_1)<k/2$ for some $v_1\in B$ (or similarly $f(v)\geq k/2$ for all $v\in B$ and $f(u_1)<k/2$ for some $u_1\in A$).
By Remarks \ref{2.2}(8), $k$ is an even number and there should exist exactly one vertex, say $u_1\in A$, with $f(u_1)>k/2$ and so $f(u_i)=k/2$ for all $2\leq i \leq m$. Then 
$$w(f)\geq (f(v_1)+f(u_1))+\sum_{i=2}^mf(u_i)=k+\sum_{i=2}^mf(u_i).$$
Hence, 
$$w(f)\geq \left\lbrace\begin{array}{ll}
3k/2; & m=2, k \textrm{ is even} \\ 
2k; & m>2, k \textrm{ is even}
\end{array} \right.$$ 

Case III. $f(u_1)<k/2$ and $f(v_1)<k/2$ for some $u_1\in A$ and $v_1\in B$. Then by Remarks \ref{2.2}(8), there should exist $u_2\in A$ and $v_2\in B$ such that $f(u_2)>k/2$ and $f(v_2)>k/2$. Then
$$w(f)\geq (f(u_1)+f(v_2)) +(f(v_1)+f(u_2))=2k.$$

Now if $m=2$ and $k$ is an even number, then by setting $f(u_1)=f(u_2)=f(v_1)=k/2$  and $f(v_i)=0$ for all $2\leq i \leq n$ we get to a perfect strong Roman $k$-dominating function. This completes the proof.
\end{proof}

The following result, determines the strong Roman $k$-domination number of complete bipartite graphs explicitly.
\begin{theorem}\label{thm2}
$$\gamma_k^s(K_{m,n})=\left\lbrace \begin{array}{ll}
k; & m=1, n\geq 1 \\ 
3; & k=2, m=2, n\geq 2 \\
k+1; &  k \textrm{ is odd, } m=2, n\geq 2 \\
k+2; & 2<k \textrm{ is even}, m=2, n\geq 2 \\
\frac{3k+3}{2}; & 1<k \textrm{ is odd}, m=3, n\geq 3 \\  
\frac{3k+4}{2}; & 2<k \textrm{ is even}, m=3, n\geq 3 \\
2k; & \textrm{ else.}
\end{array} \right.$$
\end{theorem}
\begin{proof}
We separate the proof in different situations as follows:

\textbf{A.  ${\bf m=1}$ and ${\bf n\geq 1}$.} 

By assigning $k$ to the vertex of $A$ and zero to other vertices we get to a strong Roman $k$-dominating function. So $\gamma_k^s (K_{1,n})=k$.  

Now suppose that  $m,n\geq 2$ and $f$ is a strong Roman $k$-dominating function.

\textbf{B.  ${\bf k}$ is even, ${\bf m=2}$ and ${\bf n\geq 2}$.}

Case I. $f(u)\geq k/2$ for all $u\in V$. Then since $|V|\geq 4$, 
$$w(f)\geq 2k.$$

Case II. $f(u)\geq k/2$ for all $u\in A$ and $f(v_1)<k/2$ for some $v_1\in B$ (or similarly $f(v)\geq k/2$ for all $v\in B$ and $f(u_1)<k/2$ for some $u_1\in A$).
By Remarks \ref{2.2}(8),  there should exist at least one vertex, say $u_1\in A$ with $f(u_1)>k/2$. Now if $f(u_2)=k/2$, then 
$$w(f)\geq (f(v_1)+f(u_1))+f(u_2)\geq k+k/2=3k/2.$$
Else, $f(u_2)>k/2$. So
$$w(f)\geq f(v_1)+f(u_1)+f(u_2)\geq f(u_1)+f(u_2)\geq 2(k/2+1)=k+2.$$

Case III. $f(u_1)<k/2$ and $f(v_1)<k/2$ for some $u_1\in A$ and $v_1\in B$. Then 
$$w(f)\geq (f(u_1)+\sum_{v\in B, f(v)>k/2}f(v)) + (f(v_1)+\sum_{u\in A, f(u)>k/2}f(u))\geq 2k.$$

Thus $\gamma_k^s(K_{2,n})\geq \min\{3k/2, k+2\}$. By setting $f(u_1)=2, f(u_2)=1, f(v_i)=0$, we get to 
$$\gamma_2^s(K_{2,n})=3.$$
 Also by setting $f(u_1)=f(u_2)=k/2+1, f(v_i)=0$ we get to 
$$\gamma_k^s(K_{2,n})=k+2,$$
 for even integers $k$ with $k>2$.

\textbf{C. ${\bf k}$ is odd, ${\bf m=2}$ and ${\bf n\geq 2}$.} 

Case I. $f(u)\geq (k+1)/2$ for all $u\in A$ or $f(v)\geq (k+1)/2$ for all $v\in B$. Then since $n\geq m=2$,
$$w(f)\geq k+1.$$

Case II.  $f(u_1)<k/2$ and $f(v_1)<k/2$ for some $u_1\in A$ and $v_1\in B$. Then in view of Remarks \ref{2.2}(8), we should have $f(u_2)\geq (k+1)/2$. So
\begin{align*}
w(f)&\geq f(u_2)+(f(u_1)+\sum _{v\in B, f(v)>k/2} f(v))\\
&\geq (k+1)/2+k\\
&=(3k+1)/2\\
&\geq k+1.
\end{align*}

The above cases show that $\gamma_k^s(K_{2,n})\geq k+1$. By assigning $(k+1)/2$ to vertices of $A$ and zero to vertices of $B$, we get to a strong Roman $k$-dominating function of $K_{2,n}$ with weight $k+1$. Hence we have
$$\gamma_k^s(K_{2,n})=k+1.$$

\textbf{D. ${\bf k}$ is odd, ${\bf m=3}$ and ${\bf n\geq 3}$.} 

Case I.  $f(u)\geq (k+1)/2$ for all $u\in A$ or $f(v)\geq (k+1)/2$ for all $v\in B$. Then since $n\geq m=3$, 
$$w(f)\geq (3k+3)/2.$$

Case II.  $f(u_1)<k/2$ and $f(v_1)< k/2$ for some $u_1\in A$ and $v_1\in B$. Then 
$$w(f)\geq (f(u_1)+\sum_{v\in B,f(v)>k/2}f(v))+(f(v_1)+\sum _{u\in A, f(u)>k/2}f(u))\geq 2k.$$
Hence $\gamma_k^s(K_{3,n})\geq (3k+3)/2$, for odd integers $k>1$ and $\gamma_k^s(K_{3,n})\geq 2k$ for $k=1$. Now by assigning $(k+1)/2$ to vertices of $A$ and zero to vertices of $B$, we get to a strong Roman $k$-dominating function with weight $(3k+3)/2$. Thus
$$\gamma_k^s(K_{3,n})=\frac{3k+3}{2},$$
for $k>1$ and 
$$\gamma_1^s(K_{3,n})=2.$$

\textbf{E. ${\bf k}$ is even, ${\bf m=3}$ and ${\bf n\geq 3}$.}

Case I.  $f(u)\geq k/2$ for all $u\in V$. Then since $|V|\geq 6$, 
$$w(f)\geq 3k.$$

Case II. $f(u_1)<k/2$ and $f(v_1)<k/2$ for some $u_1\in A$ and $v_1\in B$. Then we have
$$w(f)\geq (f(u_1)+\sum_{v\in B, f(v)>k/2}f(v))+(f(v_1)+\sum_{u\in A,f(u)>k/2}f(u))\geq 2k.$$

Case III. $f(u_1)<k/2$ for some $u_1\in A$ and $f(v)\geq k/2$ for all $v\in B$ (or similarly $f(v_1)<k/2$ for some $v_1\in B$ and $f(u)\geq k/2$ for all $u\in A$). Then, by Remarks \ref{2.2}(8) we should have $f(v_1)\geq k/2+1$ for some $v_1\in B$. Now if there are two vertices $v_1, v_2$ in $B$ with $f(v_1), f(v_2)\geq k/2+1$, then 
$$w(f)\geq f(v_1)+f(v_2)+f(v_3)\geq (k/2+1)+(k/2+1)+k/2=\frac{3k+4}{2}.$$
Else, $f(u_1)+f(v_1)\geq k$ and $f(v)=k/2$ for all $v\in B\setminus \{v_1\}$. Hence
$$w(f)\geq (f(u_1)+f(v_1))+f(v_2)+f(v_3)\geq k+k/2+k/2=2k.$$
Therefore $\gamma_k^s(K_{3,n})\geq \min \{\frac{3k+4}{2}, 2k\}$. By setting $f(u_1)=f(u_2)=k/2+1, f(u_3)=k/2, f(v_i)=0$ for all $1\leq i \leq n$, we get to a strong Roman $k$-dominating function $f$ with $w(f)=(3k+4)/2$. Thus
$$\gamma_2^s(K_{3,n})=4, \gamma_k^s(K_{3,n})=\frac{3k+4}{2}.$$

\textbf{F. ${\bf n\geq m\geq 4}$.}

We proceed when $k$ is an even number. The result can be gained similarly for odd integers $k$. In view of Remarks \ref{2.2}(8) the following cases may occur.

Case I.  $f(u)\geq k/2$ for all $u\in A$ or $f(v)\geq k/2$ for all $v\in B$. Then since $n\geq m\geq 4$, 
$$w(f)\geq 2k.$$

Case II.  $f(u_1)<k/2$ and $f(v_1)< k/2$ for some $u_1\in A$ and $v_1\in B$. Now, if $f(u_2), f(u_3)\geq k/2$ for some $u_2, u_3\in A$ (or similarly $f(v_2), f(v_3)\geq k/2$ for some $v_2, v_3\in B$), then 
$$w(f)\geq (f(u_1)+\sum_{v\in B,f(v)>k/2}f(v))+f(u_2)+f(u_3)\geq k+k/2+k/2=2k.$$

Case III. $f(u_1)\geq k/2, f(v_1)\geq k/2$ for some $u_1\in A$ and $v_1\in B$ and $f(u)<k/2$ for all $u\in V\setminus\{u_1, v_1\}$. Then by Remarks \ref{2.2}(8) we should have $f(u_1), f(v_1)>k/2$. Thus
$$f(u_2)+f(v_1)\geq k, f(u_3)+f(v_1)\geq k,f(u_4)+f(v_1)\geq k,$$
$$f(v_2)+f(u_1)\geq k, f(v_3)+f(u_1)\geq k, f(v_4)+f(u_1)\geq k.$$
Hence by summing the above inequalities, we have $3w(f)\geq 6k$. So
$w(f)\geq 2k$, and hence 
$$\gamma_k^s(K_{m,n})=2k.$$

\end{proof}

The third result of this section, qualifies the perfect Roman $k$-domination number of complete bipartite graphs in various situations as follows.
\begin{theorem}\label{thm3}
$$\gamma_k^p(K_{m,n})=\left\lbrace \begin{array}{ll}
k; & m=1, n\geq 1 \\ 
k; & k \textrm{ is even}, m=2, n\geq 2 \\
\frac{3k+1}{2} \textrm{ or } k\frac{3n-2}{2n-1}; &  k \textrm{ is odd, } m=2, n\geq 2 \\
\frac{2mn-m-n}{mn-1}k \textrm{ or } \frac{4m-3}{2m}k \leq \gamma_k^p(K_{m,n}); & k \textrm{ is even}, n\geq m\geq 3 \\
\frac{2mn-m-n}{mn-1}k \textrm{ or } \frac{(4m-3)k+1}{2m} \leq \gamma_k^p(K_{m,n}); & k \textrm{ is odd}, n\geq m\geq 3 \\
\end{array} \right.$$
In the case $n\geq m\geq 3$, the value of $\gamma_k^p(K_{m,n})$ can be gained by solving the systems \ref{1.1.1} or \ref{1.2} or \ref{1.3} in the proof if they are consistent and else it equals to $2k$.
\end{theorem}
\begin{proof}
We separate the proof in different situations as follows:

\textbf{A.  ${\bf m=1}$ and ${\bf n\geq 1}$.} 

By assigning $k$ to the vertex of $A$ and zero to other vertices we get to a perfect Roman $k$-dominating function with weight $k$. So $\gamma_k^p (K_{1,n})=k$. 

\textbf{B. ${\bf k}$ is even, ${\bf m=2}$ and ${\bf n\geq 2}$.} 

By assigning $k/2$ to the vertices of $A$ and zero to the vertices of $B$ we get to a perfect Roman $k$-dominating function with weight $k$. So $\gamma_k^p (K_{2,n})=k$ for even integers $k$.

Now in other cases, suppose that $f$ is a perfect Roman $k$-dominating function. 

\textbf{C. ${\bf k}$ is odd, ${\bf m=2}$ and ${\bf n\geq 2}$.}

Case I. $f(u)\geq (k+1)/2$ for all $u\in V$. Then clearly since $|V|\geq 4$,
$$w(f)\geq 4(k+1)/2>(3k+1)/2.$$

Case II. $f(u)\geq (k+1)/2$ for all $u\in A$ and $f(v_1)<k/2$ for some $v_1\in B$ (or similarly $f(u_1)<k/2$ for some $u_1\in A$ and $f(v)\geq (k+1)/2$ for all $v\in B$). Then since $|A|=2$,
$k+1\leq f(v_1)+\sum_{u\in A}f(u)=k,$ which is a contradiction.

Case III. $f(u_1)<k/2$, $f(v_1)< k/2$ and $f(v_2)\geq (k+1)/2$ for some $u_1\in A$ and $v_1,v_2\in B$ (or similarly $f(u_1)<k/2$, $f(v_1)< k/2$ and $f(u_2)\geq (k+1)/2$ for some $u_1, u_2\in A$ and  $v_1\in B$). Then 
$$w(f)\geq (f(v_1)+\sum_{u\in A}f(u))+f(v_2)\geq k+(k+1)/2=(3k+1)/2.$$

Case IV.  $f(u)<k/2$ for all $u\in V$. Then since 
$$f(u_1)+\sum_{v\in B}f(v)=k=f(u_2)+\sum_{v\in B}f(v),$$
we should have $f(u_1)=f(u_2)=i$ for some integer $0\leq i \leq (k-1)/2$. Similarly, we should have $f(v)=j$ for all $v\in B$, where $j$ is an integer with $0\leq j \leq (k-1)/2$. So we have the following system of linear equations.
\begin{equation}\label{1.1}
\left\lbrace \begin{array}{l}
i+nj=k \\ 
j+2i=k
\end{array}\right.
\end{equation}
Solving the above system implies that
$$i=\frac{k(n-1)}{2n-1}, j=\frac{k}{2n-1},$$
when $i,j$ are integers with $0\leq i,j \leq (k-1)/2$. In particular we should have $2n-1 | k$. Therefore
in this case $w(f)=2i+nj=k\frac{3n-2}{2n-1}<\frac{3k+1}{2}$.

On the other hand, if we set 
$$f(u_1)=(k+1)/2, f(u_2)=(k-1)/2, f(v_1)=(k+1)/2$$
and
$$f(v)=0; \ \forall v\in B\setminus\{v_1\},$$
then $f$ is a perfect $k$-dominating function with weight $(3k+1)/2$. So $\gamma_k^p(K_{m,n})\leq (3k+1)/2$. 

The above arguments show that 
$$\gamma_k^p(K_{m,n})=(3k+1)/2,$$
 unless $k, n$ are integers such that the system of equations (\ref{1.1}) is consistent, which then $$\gamma_k^p(K_{m,n})=k\frac{3n-2}{2n-1}.$$

\textbf{D. ${\bf n\geq m \geq 3}$.} 

Case I. $f(u)\geq k/2$ for all $u\in V$. Then since $|V|\geq 6$, 
$$w(f)\geq 3k.$$

Case II. $f(u_1)<k/2$ for some $u_1\in A$ and $f(v)\geq k/2$ for all $v\in B$ (or similarly $f(v_1)<k/2$ for some $v_1\in B$ and $f(u)\geq k/2$ for all $u\in A$). Then since $|B|\geq 3$, 
$$k=f(u_1)+\sum_{v\in B}f(v)\geq 3k/2,$$
which is a contradiction.

Case III. $f(u_1)<k/2$, $f(u_2), f(u_3)\geq k/2$ and $f(v_1)<k/2$ for some $u_1, u_2, u_3\in A$ and $v_1\in B$ (or similarly $f(u_1)<k/2$, $f(v_1)<k/2$ and $f(v_2), f(v_3)\geq k/2$ for some $u_1\in A$ and $v_1, v_2, v_3\in B$). Then 
$$w(f)\geq f(u_2)+f(u_3)+(f(u_1)+\sum_{v\in B}f(v))\geq k/2+k/2+k=2k.$$
(Of course if $k$ is odd, this case is impossible, because $f(u_2), f(u_3)\geq (k+1)/2$ and so $k=f(v_1)+\sum_{u\in A}f(u)\geq k+1$ which is a contradiction.)

Case IV. $f(u_1)\geq k/2$ for some $u_1\in A$ and $f(u)<k/2$ for all $u\in V\setminus \{u_1\}$ (or similarly $f(v_1)\geq k/2$ for some $v_1\in B$ and $f(u)<k/2$ for all $u\in V\setminus \{v_1\}$). Then 
$$f(u_i)+\sum_{v\in B}f(v)=k, \ \forall i \textrm{ with } 2\leq i \leq m,$$
$$f(v_i)+\sum_{u\in A}f(u)=k, \ \forall i \textrm{ with } 2\leq i \leq m.$$
Summing up the above equalities shows that
$$mw(f)\geq 2(m-1)k+f(u_1)\geq 2(m-1)k+k/2=\frac{4m-3}{2}k,$$
when $k$ is even, and
$$mw(f)\geq 2(m-1)k+f(u_1)\geq 2(m-1)k+(k+1)/2=\frac{(4m-3)k+1}{2},$$
when $k$ is odd.
Hence
\begin{align*}
w(f)\geq \left\lbrace\begin{array}{c}
\frac{4m-3}{2m}k; \ \ k \textrm{ is even} \\ 
\frac{(4m-3)k+1}{2m}; \ \ k \textrm{ is odd}
\end{array} \right.
\end{align*}
Note that in this case we should have 
$$f(u)=i, f(v)=j, \ \forall u \in A\setminus \{u_1\}, v\in B,$$
for some $0\leq i,j <k/2$. Hence 
\begin{equation}\label{1.1.1}
\left\lbrace \begin{array}{l}
j+(m-1)i+f(u_1)=k\\ 
i+nj=k \\ 
0\leq i,j < k/2\\ 
 f(u_1)=k-(m-1)i-j\geq k/2.
\end{array}\right.
\end{equation}
So $w(f)=k+(n-1)j$. Note that for example when $k=10$ and $m=n=3$ by setting $f(u_1)=5, f(u_2)=f(u_3)=1, f(v_1)=f(v_2)=f(v_3)=3$ we get to a perfect Roman $10$-dominating function with weight $16$. But when $k=4$ and $m=n=3$, the system \ref{1.1.1} is inconsistent.

Case V.  $f(u_1)\geq k/2, f(v_1)\geq k/2$ for some $u_1\in A$ and $v_1\in B$ and $f(u)<k/2$ for all $u\in V\setminus \{u_1,v_1\}$. Then 
$$f(u_i)+\sum_{v\in B}f(v)=k, \ \forall i \textrm{ with } 2\leq i \leq m,$$
$$f(v_i)+\sum_{u\in A}f(u)=k, \ \forall i \textrm{ with }  2\leq i \leq m.$$ 
Summing up the above equalities shows that
$$mw(f)\geq 2(m-1)k+f(u_1)+f(v_1)\geq 2(m-1)k+k/2+k/2=(2m-1)k,$$
when $k$ is even and 
$$mw(f)\geq 2(m-1)k+f(u_1)+f(v_1)\geq 2(m-1)k+(k+1)/2+(k+1)/2=(2m-1)k+1,$$
when $k$ is odd. Hence
\begin{align*}
w(f)\geq \left\lbrace\begin{array}{c}
\frac{2m-1}{m}k; \ \ k \textrm{ is even} \\ 
\frac{(2m-1)k+1}{m}; \ \ k \textrm{ is odd}
\end{array} \right.
\end{align*}
Note that in this case we should have 
$$f(u)=i, f(v)=j, \ \forall u \in A\setminus \{u_1\}, v\in B\setminus \{v_1\},$$
for some $0\leq i,j <k/2$. Hence 
\begin{equation}\label{1.2}
\left\lbrace \begin{array}{l}
j+(m-1)i+f(u_1)=k\\ 
i+(n-1)j+f(v_1)=k \\ 
0\leq i,j < k/2\\ 
 f(u_1)=k-(m-1)i-j\geq k/2\\
f(v_1)=k-(n-1)j-i\geq k/2.
\end{array}\right.
\end{equation}
So $w(f)=2k-(i+j)$. 
As you may see in this case the precise value of $w(f)$ is related to values of $m, n$ and $k$, too. For instance when $k=10, m=3$ and $n=3$ or $4$, by setting $i=2, j=1$ we have $w(f)=17$ which is the least integer greater than $\frac{2m-1}{m}k$ in this case but the perfect Roman $10$-dominating function defined in Case IV had less weight. As another example when $m=n=3, k=4$, the minimum weight in this case is when $i=1, j=0$ and so it is $7$, while the system \ref{1.1.1} doesn't have any solution and the next case give a $\gamma_k^p$-function with weight $6$. 

Case VI. $f(u)<k/2$ for all $u\in V$. Then since $f(u)+\sum_{v\in B}f(v)=k$ for all $u\in A$, we have $f(u)=i$ for some $0\leq i <k/2$. Similarly for all $v\in B$, $f(v)=j$ for some $0\leq j <k/2$. Hence we have the following system of linear equations.
\begin{equation}\label{1.3}
\left\lbrace \begin{array}{l}
i+nj=k \\ 
j+mi=k
\end{array}\right.
\end{equation}
Solving the above system implies that
$$i=\frac{n-1}{mn-1}k, j=\frac{m-1}{mn-1}k.$$
Therefore $w(f)=mi+nj=\frac{2mn-m-n}{mn-1}k$ in this case.

Note that for all positive integers $m,n$ and $k$ we have
$$\frac{2mn-m-n}{mn-1}k<\frac{2m-1}{m}k<2k, \frac{4m-3}{2m}k<\frac{2m-1}{m}k<2k.$$

Hence, a $\gamma_k^p$-function can be gained by solving one of the systems \ref{1.1.1} or \ref{1.2} or \ref{1.3}. If they don't have any solution, then we should have $\gamma_k^p(K_{m,n})=2k$. 
\end{proof}

In the last main result of our note, we specify the value of Roman $k$-domination number of complete bipartite graph as follows.
\begin{theorem}\label{thm4}
$$\gamma_k(K_{m,n})=\left\lbrace \begin{array}{ll}
k; & m=1, n\geq 1 \\ 
k; & k \textrm{ is even}, m=2, n\geq 2 \\
k+1; &  k \textrm{ is odd, } m=2, n\geq 2 \\
\frac{3k}{2}; & k \textrm{ is even}, m=3, n\geq 3 \\
\frac{3k+1}{2}; & k=1,3 , m=3, n\geq 3 \\  
\frac{3k+1}{2}; & k \textrm{ is odd}, m=n=3 \\
\frac{3k+3}{2}; & 5\leq k \textrm{ is odd}, m=3, n\geq 6 \\
\frac{3k+3}{2}; & 11\leq k \textrm{ is odd}, m=3, n=5 \\
\frac{3k+3}{2}; & 17\leq k \textrm{ is odd}, m=3, n=4 \\
\frac{3k+1}{2} \textrm{ or } \frac{3k+3}{2}; & 5\leq k \leq 15, k \textrm{ is odd}, m=3, n=4 \\
\frac{3k+1}{2} \textrm{ or } \frac{3k+3}{2}; & 5\leq k \leq 9, k \textrm{ is odd}, m=3, n=5 \\
\frac{2m}{m+1}k \leq \gamma_k(K_{m,n})\leq 2k; & n\geq m\geq 4 
\end{array} \right.$$
\end{theorem}
\begin{proof}
We separate the proof in different situations as follows:

\textbf{A.  ${\bf m=1}$ and ${\bf n\geq 1}$.} 

By assigning $k$ to the vertex of $A$ and zero to other vertices we get to a Roman $k$-dominating function with weight $k$. So $\gamma_k (K_{1,n})=k$. 

\textbf{B. ${\bf k}$ is even, ${\bf m=2}$ and ${\bf n\geq 2}$.} 

By assigning $k/2$ to the vertices of $A$ and zero to other vertices we get to a Roman $k$-dominating function with weight $k$. So $\gamma_k (K_{2,n})=k$ for even integers $k$. 

Hereafter suppose that $f$ is a Roman $k$-dominating function.

\textbf{C. ${\bf k}$ is odd, ${\bf m=2}$ and ${\bf n\geq 2}$.} 

Case I. $f(u)\geq (k+1)/2$ for all $u\in A$ (or similarly $f(v)\geq (k+1)/2$ for all $v\in B$). Then since $|A|\geq 2$, clearly we have $w(f)\geq k+1$.

Case II.  $f(u_1)<k/2$ and $f(u_2)\geq (k+1)/2$ (or similarly $f(v_1)<k/2$ and $f(v_2)\geq (k+1)/2$ for some $v_1, v_2\in B$). Then 
\begin{align*}
w(f)&\geq f(u_2)+(f(u_1)+\sum_{v\in B}f(v))\\
&\geq (k+1)/2+k\\
&\geq k+1.
\end{align*}

Case III.  $f(u)<k/2$ for all $u\in V$. Therefore
$$f(u_1)+\sum_{v\in B}f(v)\geq k, f(u_2)+\sum_{v\in B}f(v)\geq k,$$
$$f(v_1)+\sum_{u\in A}f(u)\geq k, f(v_2)+\sum_{u\in A}f(u)\geq k,$$
for some $v_1, v_2\in B$.
Summing up the above inequalities implies that $3w(f)\geq 4k$, which ensures that $w(f)\geq k+1$, since $w(f)$ is an integer.

The above cases show that $\gamma_k(K_{2,n})\geq k+1$. By assigning $(k+1)/2$ to the vertices of $A$ and zero to the vertices of $B$, we get to a Roman $k$-dominating function of $K_{m,n}$ with weight $k+1$. Hence $\gamma_k(K_{2,n})=k+1$,
in this case.

\textbf{D. ${\bf n\geq m\geq 3}$.}

Case I.  $f(u)\geq k/2$ for all $u\in A$ (or similarly $f(v)\geq k/2$ for all $v\in B$). Then 
$$w(f)\geq mk/2\geq \frac{2m}{m+1}k.$$
(of course if $k$ is odd, we even have  $w(f)\geq m(k+1)/2$.)

Case II.  $f(u_1)<k/2$ and $f(u_2), f(u_3)\geq k/2$ for some $u_1, u_2, u_3\in A$ (or similarly  $f(v_1)<k/2$ and  $f(v_2), f(v_3)\geq k/2$ for some $v_1, v_2,v_3\in B$). Then 
$$w(f)\geq (f(u_1)+\sum_{v\in B}f(v))+f(u_2)+f(u_3)\geq k+k/2+k/2=2k.$$

Case III. $f(u_1)\geq k/2$ and $f(u)<k/2$ for all $u\in V\setminus \{u_1\}$ (or similarly $f(v_1)\geq k/2$ and $f(u)<k/2$ for all $u\in V\setminus \{v_1\}$). Then 
$$f(u_i)+\sum_{v\in B}f(v)\geq k, \ \forall i \textrm{ with }   2\leq i \leq m,$$
$$f(v_i)+\sum_{u\in A}f(u)\geq k, \ \forall i \textrm{ with }  1\leq i \leq m.$$ 
Note that since there exists $u_i\in A$ with $f(u_i)<k/2$ and $f(u_i)+\sum_{v\in B}f(v)\geq k$, we should have $\sum_{v\in B}f(v)\geq k/2$. So, summing up the above inequalities implies $(m+1)w(f)\geq (2m-1)k+f(u_1)+\sum_{v\in B}f(v)\geq 2mk$. Hence
$$w(f)\geq \frac{2m}{m+1}k.$$

Case IV. $f(u_1)\geq k/2, f(v_1)\geq k/2$ for some $u_1\in A, v_1\in B$ and $f(u)<k/2$ for all $u\in V\setminus \{u_1, v_1\}$. Then
$$f(u_i)+\sum_{v\in B}f(v)\geq k, \ \forall i \textrm{ with }  2\leq i \leq m,$$
$$f(v_i)+\sum_{u\in A}f(u)\geq k, \ \forall i \textrm{ with }  2\leq i \leq m.$$
Summing up the above inequalities implies 
$$mw(f)\geq (2m-2)k+f(u_1)+f(v_1)\geq (2m-1)k.$$
So
$$w(f)\geq \frac{2m-1}{m}k> \frac{2m}{m+1}k.$$

Case V. $f(u)<k/2$ for all $u\in V$. Then 
$$f(u_i)+\sum_{v\in B}f(v)\geq k, \ \forall i \textrm{ with }  1\leq i \leq m,$$
$$f(v_i)+\sum_{u\in A}f(u)\geq k, \ \forall i \textrm{ with }  1\leq i \leq m.$$
Summing up the above inequalities implies $(m+1)w(f)\geq 2mk$. So
$$w(f)\geq \frac{2m}{m+1}k.$$

The above cases show that $\gamma_k(K_{m,n})\geq \frac{2m}{m+1}k$ and so 
\begin{equation}
\frac{2m}{m+1}k \leq \gamma_k(K_{m,n})\leq 2k.
\end{equation}

Note that for special values of $k$, we may have $w(f)=\frac{2m}{m+1}k$. For instance, when $k=(m+1)r$ for some integer $r$, by assigning $r$ to each vertex of $K_{m,m}$ we get to a Roman $k$-dominating function with weight $\frac{2m}{m+1}k$.

If $m=3$, then $(5)$ shows that $w(f)\geq 3k/2$. Also if $k$ is even, then by assigning $k/2$ to the vertices of $A$ and zero to the vertices of $B$, one can get to a  Roman $k$-dominating function of $K_{m,n}$ with weight $3k/2$. So 
$$\gamma_k(K_{3,n})=3k/2,$$
when $n\geq 3$ and $k$ is even.

If $m=3$ and $k=1$, then since $w(f)$ is an integer, $(5)$ implies that 
$$\gamma_k(K_{m,n})=(3k+1)/2=2k.$$

If $m=3$ and $1<k$ is odd, then by assigning $(k+1)/2$ to the vertices of $A$ and zero to the vertices of $B$, we have a Roman $k$-dominating function with weight $3(k+1)/2$. Thus $\gamma_k(K_{m,n}) \leq (3k+3)/2 $. So, by $(5)$  for $m=3$, we have 
$$3k/2 \leq \gamma_k(K_{m,n})\leq (3k+3)/2.$$
 Hence since $w(f)$ is an integer, we have $\gamma_k(K_{m,n})=(3k+1)/2$ or $\gamma_k(K_{m,n})=(3k+3)/2$. 

When $m=n=3$, by setting
 $$f(u_1)=(k-3)/4, f(u_2)=f(u_3)=f(v_1)=f(v_2)=f(v_3)=(k+1)/4,$$
 when $k\overset{4}{\equiv} 3$  and 
$$f(u_1)=f(v_1)=(k+3)/4, f(u_2)=f(u_3)=f(v_2)=f(v_3)=(k-1)/4,$$
 when $k\overset{4}{\equiv}1$,  we get to a Roman $k$-dominating function with weight $(3k+1)/2$. Thus 
$$\gamma_k(K_{3,3})=(3k+1)/2.$$
Also, when $k=3$ and $n\geq m=3$, if $f(u)=1$ for all $u\in A$, $f(v_1)=2$ for some $v_1\in B$ and $f(v)=0$ for all $v\in B\setminus \{v_1\}$, then $f$ is a Roman $3$-dominating function with weight $5$ and so 
$$\gamma_3(K_{3,n})=(3k+1)/2.$$
Now suppose that $k>3$ and $n>m=3$. Also assume that there exists a Roman $k$-dominating function, say $f$, with $w(f)=(3k+1)/2$. As the above cases illustrate, it may exists in Cases III or V. 

Firstly in Case III, suppose that $f(v_n)\geq (k+1)/2$ and $f(u)<k/2$ for all $u\in V\setminus \{v_n\}$. Then
$$f(u_i)+\sum_{v\in B}f(v)\geq k, \ \forall i \textrm{ with }   1\leq i \leq 3, $$
$$f(v_i)+\sum_{u\in A}f(u)\geq k, \ \forall i \textrm{ with }   1\leq i \leq n-1.$$
Summing the above inequalities implies that
$$3w(f)+(w(f)-f(v_n))+(n-4)\sum_{u\in A}f(u)\geq (n+2)k.$$
Since $f(u_1)\leq (k-1)/2$, $f(u_1)+\sum_{v\in B}f(v)\geq k$ ensures that $\sum_{v\in B}f(v)\geq (k+1)/2$. Hence
\begin{align*}
nw(f)&\geq (n+2)k+f(v_n)+(n-4)\sum_{v\in B}f(v)\\
&\geq (n+2)k+(k+1)/2+(n-4)(k+1)/2\\
&=(3nk+n+k-3)/2.
\end{align*}
Therefore $k>3$ implies
$$\frac{3k+1}{2}=w(f)\geq \frac{3k+1}{2}+\frac{k-3}{2n}> \frac{3k+1}{2},$$
which is a contradiction.

Now suppose that in Case III, $f(u_3)\geq (k+1)/2$ and $f(u)<k/2$ for all $u\in V\setminus \{v_n\}$. Then
$$f(u_i)+\sum_{v\in B}f(v)\geq k, \ \forall i \textrm{ with }   1\leq i \leq 2, $$
$$f(v_j)+\sum_{u\in A}f(u)\geq k, \ \forall j \textrm{ with }   1\leq j \leq n.$$
Summing the above inequalities for each $1\leq i \leq 2$ and $1\leq j \leq n$ implies that $f(u_i)+f(v_j)+w(f)\geq 2k$. So, $w(f)=(3k+1)/2$ yields to
\begin{equation}\label{k-1}
f(u_i)+f(v_j)\geq 2k-(3k+1)/2=(k-1)/2. 
\end{equation}
Hence if $y=\min\{f(v) \ | \ v\in B\}$, then
\begin{align*}
(3k+1)/2=w(f)&=(f(u_1)+f(v_1))+(f(u_2)+f(v_2))+f(u_3)+\sum_{i=3}^nf(v_i)\\
&\geq (k-1)/2+(k-1)/2+(k+1)/2+(n-2)y\\
&=(3k-1)/2+(n-2)y.
\end{align*}
Thus $y\leq 1/(n-2)$, which implies that $y=0$, since $n>3$. Now, (\ref{k-1}) and $f(u_1), f(u_2)\leq (k-1)/2$ yield to $f(u_1)=f(u_2)=(k-1)/2$. So since $f(u_1)+\sum_{v\in B}f(v)\geq k$, we should have $\sum_{v\in B}f(v)\geq (k+1)/2$. Therefore
\begin{align*}
(3k+1)/2=w(f)&=f(u_1)+f(u_2)+f(u_3)+\sum_{v\in B}f(v)\\
&\geq (k-1)/2+(k-1)/2+(k+1)/2+(k+1)/2\\
&=2k,
\end{align*}
which is a contradiction. So, such $f$ can't be gained by Case III.

Moreover, in Case V we have $f(u)<k/2$ for all $u\in V$. Then
$$f(u_i)+\sum_{v\in B}f(v)\geq k, \ \forall i \textrm{ with }  1\leq i \leq 3,$$
$$f(v_j)+\sum_{u\in A}f(u)\geq k,  \ \forall j \textrm{ with }  1\leq j \leq n.$$
Summing the above inequalities for each $1\leq i \leq 3$ and $1\leq j \leq n$ implies that $f(u_i)+f(v_j)+w(f)\geq 2k$. So, $w(f)=(3k+1)/2$ yields to
\begin{equation}\label{k-2}
f(u_i)+f(v_j)\geq 2k-(3k+1)/2=(k-1)/2. 
\end{equation}
Hence if $y=\min\{f(v) \ | \ v\in B\}$, then
\begin{align*}
(3k+1)/2=w(f)&=(f(u_1)+f(v_1))+(f(u_2)+f(v_2))+(f(u_3)+f(v_3))+\sum_{i=4}^nf(v_i)\\
&\geq 3(k-1)/2+(n-3)y.
\end{align*}
Thus $y\leq 2/(n-3)$. Now if $y=0$, then (\ref{k-2}) and $f(u_i)\leq (k-1)/2$ yield to $f(u_i)=(k-1)/2$ for all $i=1,2,3$. Hence
$$(3k+1)/2=w(f)=\sum_{u\in A}f(u)+\sum_{v\in B}f(v)\geq 3(k-1)/2+(k+1)/2=(4k-2)/2.$$
This implies that $k\leq 3$ which is a contradiction. Hence when $5\leq k$ is an odd number and $n\geq 6$, such $f$ can't be gained by Case V. So
$$\gamma_k(K_{3,n})=(3k+3)/2,$$
when $k\geq 5$ is an odd number, $n\geq 6$ and $m=3$.

Now if $y=1$, then (\ref{k-2}) and $f(u_i)\leq (k-1)/2$ yield to $f(u_i)\geq (k-3)/2$ for all $i=1,2,3$. Hence
$$(3k+1)/2=w(f)=\sum_{u\in A}f(u)+\sum_{v\in B}f(v)\geq 3(k-3)/2+(k+1)/2=(4k-8)/2.$$
This implies that $k\leq 9$. Hence 
$$\gamma_k(K_{3,5})=(3k+3)/2,$$
when $k\geq 11$ is an odd number.

Also, if $y=2$, then similarly we have 
$$(3k+1)/2=w(f)=\sum_{u\in A}f(u)+\sum_{v\in B}f(v)\geq 3(k-5)/2+(k+1)/2=(4k-14)/2.$$
This implies that $k\leq 15$. Hence 
$$\gamma_k(K_{3,4})=(3k+3)/2,$$
when $k\geq 17$ is an odd number.
 
 These complete the proof. Note that as discussed earlier, when either $k$ is odd with  $5\leq k \leq 15$ and $m=3, n=4$ or $k$ is odd with $5\leq k \leq 9$ and $m=3, n=5$, the precise value of $\gamma_k(K_{m,n})$ can be found by Case V and $y=\min\{f(v) \ | \ v\in B\}=1$ or $2$. For instance when $k=5, n=4, m=3$, if $f(u_1)=2$ and $f(v)=1$ for all $v\in V\setminus \{u_1\}$, then  $f$ is a Roman $5$-dominating function with weight $8$ and so 
$$\gamma_5(K_{3,4})=(3k+1)/2.$$
When $k=7, n=5, m=3$, if $f(u)=2$ and $f(v)=1$ for all $u\in A$ and $v\in B$, then  $f$ is a Roman $7$-dominating function with weight $11$ and so 
$$\gamma_7(K_{3,5})=(3k+1)/2.$$
When $k=11, n=4, m=3$, if $f(u)=3$ and $f(v)=2$ for all $u\in A$ and $v\in B$, then  $f$ is a Roman $11$-dominating function with weight $17$ and so 
$$\gamma_{11}(K_{3,4})=(3k+1)/2.$$
But when $k=5, n=5, m=3$, the above discussion shows that $y=1$ and so $f(u_i)=1$ or $2$. Since $y+\sum_{u\in A}f(u)\geq 5$ and $f(v)\geq y$ for all $v\in B$, we have
$$w(f)=\sum_{u\in A}f(u)+\sum_{v\in B}f(v)\geq 9>(3k+1)/2.$$
Hence
$$\gamma_5(K_{3,5})=(3k+3)/2.$$
Therefore, in each case it should be discussed.

\end{proof}

We end this note by the following remark.
\begin{remark}
We can use Theorems \ref{thm2} and \ref{thm4} and Remarks \ref{2.2}(7) to gain or bound (strong) Roman $k$-domination number of some other classes of graphs. For example if $G=F_{m,n}$ is a fan graph, (recall that $F_{m,n}$ is the join of the empty graph with $n$ vertices and the path graph with $m$ vertices), or $G$ is the join of two paths, join of two cycles or join of path and cycle, then clearly $K_{m,n}$ is a spanning subgraph of $G$ for some integers $m$ and $n$. So Remarks \ref{2.2}(7) implies that
$$\gamma_k(G)\leq \gamma_k(K_{m,n}), \gamma_k^s(G)\leq \gamma_k^s(K_{m,n}).$$
Therefore, Theorems \ref{thm2} and \ref{thm4} can obtain some bounds for (strong) Roman $k$-domination number of such graphs.
\end{remark}

\textbf{Acknowledgement.} The author is deeply grateful to the
referee for careful reading of the manuscript and helpful
suggestions.

\end{document}